\documentclass[10pt,a4paper]{article}

\usepackage{amsmath}
\usepackage{amsfonts}
\usepackage{amssymb}
\usepackage{amsthm}
\usepackage{ulem}
\usepackage{xcolor}
\usepackage{verbatim}
\usepackage{a4wide}

\numberwithin{equation}{section}

\newtheorem{thm}{Theorem}[section]
\newtheorem{pro}[thm]{Proposition}
\newtheorem{lem}[thm]{Lemma}
\newtheorem{cor}[thm]{Corollary}

\theoremstyle{definition}
\newtheorem*{dfn}{Definition}

\newtheorem{rem}{Remark}
\newtheorem*{ex}{Example}
\theoremstyle{plain}

\newcommand{\B}{\mathcal{B}}

\newcommand{\GB}{\Gamma \backslash \mathcal{B}}
\newcommand{\GG}{\Gamma \backslash G}

\begin{document}
\author{Shai Evra}
\title{Finite quotients of Bruhat-Tits buildings as geometric expanders}
\maketitle
\begin{abstract}
In \cite{FGLNP}, Fox, Gromov, Lafforgue, Naor and Pach, in a respond to a question of Gromov \cite{G}, constructed bounded degree geometric expanders,
namely, simplical complexes having the affine overlapping property. 
Their explicit constructions are finite quotients of $\tilde{A_d}$-buildings, for $d\geq 2$, over local fields.
In this paper, this result is extended to general high rank Bruhat-Tits buildings.
\end{abstract}

\section{Introduction}
The following definition was proposed by Gromov in \cite{G}:

\begin{dfn}
Let $X$ be a $d$-dimensional pure simplicial complex, denote by $X(0)$ its set of vertices and $X(d)$ its set of maximal simplices (chambers). 
For $0 < \epsilon \in \mathbb{R}$, say that $X$ has the $\epsilon$-overlap property or $X$ is an $\epsilon$-geometric expander, 
if for any embedding $f: X(0) \rightarrow \mathbb{R}^d$, 
there exist a point $p \in \mathbb{R}^d$ which is covered by $\epsilon$-fraction of the affine extensions of images of faces inf $X(d)$ under $f$.

A family of $d$-dimensional pure simplicial complexes are geometric expanders if they are all $\epsilon$-geometric expander for some fixed $\epsilon > 0$.
\end{dfn}

The reader can easily see that expander graphs have this property when $d=1$.
Barany's Theorem \cite{B} means that for every $d$, the $d$-dimensional complete complexes on $n$ vertices ($n \rightarrow \infty$) are geometric expanders,
and Gromov \cite{G} proved that the $d$-dimensional spherical buildings over finite fields $\mathbb{F}_q$ ($q \rightarrow \infty$) are also geometric expanders
(see also \cite{LMM} for a proof and generalizations). 
In \cite{FGLNP} Fox, Gromov, Lafforgure, Naor and Pach showed that, at least when the residue field is large enough, the finite quotients of 
Bruhat-Tits buildings of type $\tilde{A}_d$ are geometric expanders.

The goal of the current paper is to generalize this last result to other Bruhat-Tits buildings.
Here is our main result. 
\begin{thm} \label{main}
Let $d \geq 2$ and let $N_d$ be the maximal size of a spherical Coxeter group of rank $d$.
There exists $q_0 = q_0(d)$ and $\epsilon = \epsilon(d) > 0$, such that if $q \geq q_0$ is an odd prime power, then:

Let $G$ be a connected reductive group of s.s.rank $d$, with $G/Z$ almost-simple, over a local field of residue degree $q$, and let $\B$ its corresponding Bruhat-tits building.
Let $\Gamma \leq G$ be a cocompact lattice which acts type-preserving on the building and of injectivity radius $> 8\cdot d\cdot N_d$.
Then the finite quotient $\GB$ has the $\epsilon$-geometric overlap property.
\end{thm}

Our proof follows essentially the main line of \cite{FGLNP}.
There are however two delicate issues, which require some explanations. To do so let us give a sketch of the proof.
The complexes $\GB$ may be considered as partite hypergraphs $H=H(\GB) = (V_0,\ldots,V_d,E \subset \prod V_i)$.
For a partite hypergraph $H$, we define its discrepancy $Disc(H)$, to be a measure of the distance of $H$ from a random partite hypergraph (see \S 2).
The discrepancy can be bounded by the second largest eigenvalue in absolute value of certain type-induced bipartite graphs $B_i $ , $i =0,\ldots,d$, 
of $i$-type vertices versus $i$-cotype walls. 
Now the proof breaks down into two parts: 

First, one need to show that the normalized second largest eigenvalue of the type-induced bipartite graphs $B=B_i$ is small.
The original way in \cite{FGLNP} was by interpreting the adjacency operator of the 2-walk of $B$ in terms of Hecke operators of $G=PGL_{d+1}(F)$,
then using a theorem due to Oh \cite{Oh}, which is an explicit form of Kazhdan property (T), to give bounds on the spectrum of these Hecke operators.
The problem one is faced here is that for a general reductive group $G$, the bounds given by Oh's Theorem, on the Hecke operator of the 2-walk, is trivial.
To overcome this obstacle, we consider longer walks on $B$, and by interpreting their adjacency operators as a convex sum of Hecke operators,
we achieve a non-trivial bound on the spectrum of $B$.

Second, using a theorem of Pach \cite{P}, one gets a non-trivial lower bound for geometric overlap property, assuming the discrepancy is sufficiently small. 
Some care is needed: Pach's Theorem is stated in \cite{P} with the restriction $|V_0|=\ldots = |V_d|$ (which is not the case for our complexes).
However Pach's result can be generalized with essentially the same proof also if one drops this equal size restriction.

Finally, while our proof follows the main line of \cite{FGLNP}, we take the opportunity to present the proof in a more conceptual way - 
relating the geometric overlapping to the discrepancy and showing how the discrepancy is computed via induced bipartite graphs, following \cite{EGL}.

\subsection{Organization of the paper}
In Section \ref{disc} we define the notion of discrepancy of a partite hypergraph, 
then we present Pach's Theorem \cite{P} and draw from it an overlap criterion in terms of the discrepancy.
Then we use a partite mixing lemma from \cite{EGL} to bound the discrepancy by the normalized second largest eigenvalue of type-induced graphs.
In Section \ref{build} we collect some facts on the type-induced graphs of the Bruhat-Tits buildings and their finite quotients.
In Section \ref{spec} we define the Hecke operators and through them describe the adjacency operator of the type-induced graphs.
We then use Oh's Theorem \cite{Oh}, to deduce bounds on the spectrum of the Hecke operators.
After some careful computations, this will give us an effective bound on the normalized second largest eigenvalue of the type-induced graphs, 
which will imply the overlap property.
Our notations will follow \cite{EGL}.

\subsection{Acknowledgments}
I would like to thank my adviser Alex Lubotzky, for presenting me with this problem, for his guidance, encouragement and support through valuable discussions,
and for many useful corrections, comments and suggestions to this work.
I wish to thank also Shahar Mozes and Yakov Varshavsky for their help and patience in answering my many questions.
I would also like to thank the anonymous referee for his remarks and corrections, which helped simplify and shorten the paper.
I thank Jacob Fox for bring to my attention a recent result, which gives a better bound on an important constant appearing in the paper. 
I am grateful for the hospitality of the ETH Institute for Theoretical Studies.

This work is based on a M.Sc. thesis, submitted by the author, to the Institute of Mathematics at the Hebrew University of Jerusalem.

\section{Discrepancy and Pach's Theorem} \label{disc}
Let $X$ be a $d$-dimensional pure simplicial complex, and denote by $X(k)$ the collection of $k$-faces of $X$, for $k=0,\ldots,d$.
Let $X$ be equipped with a type function on its vertices, $\tau : X(0) \rightarrow [d]:=\{0,\ldots,d\}$, which is one-to-one on each chamber ($d$-face).
Let $H=H(X)$ be its associate $(d+1)$-uniform $(d+1)$-partite hypergraph, whose sets of vertices are $V_i= \tau^{-1}(\{i\}), i\in [d]$ 
and its edges are the chambers of $X$, $E = X(d)$. 

For a $(d+1)$-partite hypergraph $H=(V_0,\ldots,V_d,E)$ define its discrepancy:
\begin{equation}
Disc(H) = \max_{W_i \subseteq V_i} \vert \frac{|E(W_0,\ldots,W_d)|}{|E|} - \prod_{i=0}^d \frac{|W_i|}{|V_i|} \vert
\end{equation}
where $E(W_0,\ldots,W_d)$ is the set of edges in $E$ with a vertex in each $W_i$.

Let us now present Pach's Theorem. As mentioned in the introduction, in Pach's original Theorem there is an additional restriction that
all the point sets will be of the same size (i.e. $|P_0|=\ldots =|P_d|$), however Pach's original proof works well without this assumption.

\begin{thm}[Pach's Theorem \cite{P}]
There exist $c_d > 0$ such that: for any disjoint point sets $P_0,P_1,\ldots,P_d \subset \mathbb{R}^d$, 
there exists $p \in \mathbb{R}^d$ and $Q_i \subseteq P_i$, $\frac{|Q_i|}{|P_i|} \geq c_d$ for all $i=0,\ldots,d$,
such that every $d$-dimensional (geometric) simplex with exactly one vertex in each $Q_i$ contains $p$. 
\end{thm}

We call the constant $c_d$ appearing in the Theorem, Pach's constant.

\begin{rem}
In Pach's original paper, the constant $c_d$ is bounded from below by a function which decrease triply exponential in $d$.
However, in a recent paper by Fox, Pach and Suk \cite{FPS}, this bound was improved to a function which decrease exponentially in $d$. 
\end{rem}

From Pach's Theorem we get the following geometric overlap criterion in terms of the discrepancy.

\begin{cor} \label{over->disc}
Let $X$ be a $d$-dimensional pure simplicial complex, equipped with a type function $\tau : X(0) \rightarrow [d]$
and let $H= H(X)$ be its corresponding $(d+1)$-partite hypergraph.
Assume that $Disc(H) \leq (c_d)^{d+1} - \epsilon$ for some $\epsilon > 0$, where $c_d$ is Pach's constant.
Then $X$ has the $\epsilon$-geometric overlap property.
\end{cor}

\begin{proof}
Let $f : X(0) \rightarrow \mathbb{R}^d$ be an embedding.
For $i=0,\ldots,d$, let $V_i$ be the set of $i$-type vertices in $H$ and denote $P_i := f(V_i)$.
Then by Pach's Theorem there exists $p \in \mathbb{R}^d$ and $Q_i \subseteq P_i$, $\frac{|Q_i|}{|P_i|} \geq c_d$,
such that, if $W_i := f^{-1}(Q_i) \subseteq V_i$, then the affine extension of the image under $f$, of each edge (chamber) in $E(W_0,\ldots,W_d)$, covers $p$. 
Since $f$ is an embedding, for any $i$ we have $\frac{|W_i|}{|V_i|} = \frac{|Q_i|}{|P_i|} \geq c_d$.
Finally, by the definition of the discrepancy, we get
$$\frac{|\{\sigma \in X(d)| p \in conv(f(\sigma)) \}|}{|X(d)|} \geq \frac{|E(W_0,\ldots,W_d)|}{|E|} \geq \prod_{i=0}^d \frac{|W_i|}{|V_i|} - Disc(H) 
\geq \epsilon$$.
\end{proof}

So, if we could show that the discrepancy is sufficiently small for our complexes, we are done.
Recall that in the graph case, the Mixing Lemma can be stated in the following way:
The discrepancy of a graph is bounded by the normalized second largest eigenvalue of the adjacency operator of the graph (see \cite{EGL} Lemma 3.2 and Corollary 3.4).
We will use a form of a hypergraph Mixing Lemma, in order to bound the discrepancy by the normalized second largest eigenvalue of some induced graphs.
To this end let us first define these induced graphs.

\begin{dfn}
For any $i \in [d]$, denote $V_i=\tau^{-1}(\{i\}) \subset X(0)$ the $i$-type vertices, 
$W_i= \{F \in X(d-1)| \tau(F) =  [d] \setminus \{i\} \} \subset X(d-1)$ the $i$-cotype walls (wall = $(d-1)$-face). 
Define a bipartite graph $B_i$ whose two sets of vertices are $V_i$ and $W_i$, and there is an edge between $v \in V_i$
and $F \in W_i$ if $F \cup \{v\} \in X(d)$. We call $B_i$ the $i$-induced bipartite graph of \textit{vertices versus walls}.
\end{dfn}

Say that $X$ is \textit{type-regular}, if for any $I \subseteq J \subseteq [d]$ there is $k_{I,J} \in \mathbb{N}$
such that each $I$-type face in $X$ is contained in exactly $k_{I,J}$ faces of type $J$.
Note that the Bruhat-Tits buildings are type-regular. 

Now, the hypergraph Mixing Lemma is as follows.
\begin{pro}\cite[Corollary~3.7]{EGL}  \label{disc->B}
Let $X$ and $H$ be as above and assume further that $X$ is type-regular.
For $i=0,\ldots,d$, let $B_i$ be the the $i$-induced bipartite graph of vertices versus walls,
and let $\tilde{\lambda}(B_i)$ be the normalized second largest eigenvalue of the adjacency operator of the graph $B_i$. Then
\begin{equation}
Disc(H) \leq d \cdot \max_{i=0,\ldots,d} \tilde{\lambda}(B_i)
\end{equation}
\end{pro}

So, in conclusion, showing that the normalized second largest eigenvalue of the type-induced bipartite graphs of $X$ is small, will imply the geometric expansion.

\section{Bruhat-Tits buildings and their quotients} \label{build}
In this section we recall some properties concerning the Bruhat-Tits buildings and their quotients, 
and collect some facts on the combinatorics of their type-induced bipartite graphs.
For more on Bruhat-Tits buildings see \cite{T}, and on buildings in general, \cite{AB}.

Throughout this paper,  by a reductive group $G$ over a local field $F$, we mean that $G$ is the $F$-rational points 
of a connected reductive linear algebraic group $\tilde{G}$ defined over $F$, and $F$ is a local non-archemidean field of residue degree $q$.

\textbf{The Bruhat-Tits building} $\B =\B(G)$ associated to a reductive group $G$ over a local field, is a simplicial complex equipped with a $G$-simplicial action.
Below is a partial list of properties of the Bruhat-Tits building (see \cite{T}):
\begin{enumerate}

\item \label{sc} \textbf{Simplicial structure.} The building is a simplicial complex, along with a family of subcomplexes, called apartments.
All the apartments are isomorphic to one another and each two simplicies in the building are contained in a common apartment.
The building, and each of his apartments, are infinite, locally finite, connected, pure simplicial complexes, of dimension $d:= rank(G)$,  
equipped with a natural $(d+1)$-type-function on the vertices. 

\item \label{sym} \textbf{Symmetry.} The group $G$ acts by automorphisms on the building.
An automorphism on the building is called type-preserving, if it preserves the natural type-function mentioned earlier.
Let $G_0$ be the subgroup of type-preserving elements in $G$. Then $G_0$ acts in a type-preserving strongly-transitively way on the building (see \cite[\S~6]{AB}),
in particular $G_0$ acts transitively on faces of the same type.

\item \label{reg} \textbf{Regularity.} By property \ref{sym}, the building and each of his apartments are type-regular (see previous section).
Let $W$ be the spherical (finite) Weyl subgroup of the underlying group $\tilde{G}$ (in particular $W$ does not depends on $F$).
The regularity degrees in each apartment depends only on $W$, and not on $q$.  
In an apartment, each wall is contained in exactly $2$ chambers and each vertex is contained in at most $|W|$ chambers.
In the building, each wall is contained in exactly $q+1$ chambers and each vertex is contained in at most $q^{|W|}$ chambers.

\item \label{geo} \textbf{Geometry.} Each apartment is a tessellation of the Euclidean space $\mathbb{R}^d$, and it inherits the structure of a Euclidean metric space.
The building is a complete CAT(0) metric space, in particular contractible. Each apartment is convex inside the building, 
i.e. if an apartment $A$ contains both $\sigma_1$ and $\sigma_2$ then it contains any minimal gallery between them. 
\end{enumerate}

\begin{rem}
Recall that a Bruhat-Tits building is equipped with a natural type-function.
At least one of the types given by the type function, is what is called a special-type,
which means that stabilizers of vertices of this type satisfy the Cartan and Iwasawa decompositions.
These special-types, will be important for us later, since the play a crucial role in Oh's Theorem.
\end{rem}

\begin{dfn}  \label{N_d}
For any $d \in \mathbb{N}$, define $N_d$ to be the maximal size of a finite Coxeter group of rank $d$ (see \cite[\S 1.3]{AB} for a complete classification of such groups).
Note that $N_d$ depends only on $d$ and not on $q$. 

Let $v$ be a vertex in the building and $A$ an apartment containing it. By property \ref{reg}, 
the number of chambers in the apartment $A$ containing $v$ is at most $N_d$, and the number of chambers in the building containing $v$ is at most $q^{N_d}$.
\end{dfn}

\textbf{A finite quotient} of the Bruhat-Tits building, is a quotient of the building modulo the action of a cocompact lattice $\Gamma \leq G$, denoted $\GB$.
If $\Gamma \subset G_0$, i.e. it is type-preserving, then the quotient $\GB$ inherits the type function coming from the building.
We impose on $\B$ the graph metric of its $1$-skeleton, $dist_skl$, and we define the injectivity radius of $\Gamma$ to be 
$r(\Gamma) = \min_{x \in \B(0), 1 \ne \gamma \in \Gamma} dist_{skl}(\gamma.x,x)$.
Note that in a ball of radius $ < \frac{r(\Gamma)}{2} $, the quotient $\GB$ looks exactly like the building $\B$.
So, any local property of the building, holds for our quotients (assuming the injectivity radius is large enough).

\subsection{Combinatorics of the type-induced bipartite graphs}
Let us now collect some combinatorial facts on the type-induced bipartite graphs of the building, which we shall use in the next section.

\begin{dfn}
Let $i\in [d]$ be a fixed type, and let $B=B_i$ be the $i$-induced bipartite graph of vertices versus walls in the building $\B$. 
Denote by $dist_B$, the graph metric of $B$. Let $n \in \mathbb{N}$.

A \textit{$2n$-walk} in $B$ is a sequence $(v_0,v_1,\ldots,v_n)$ of $i$-type vertices such that for each $i=0,\ldots,n-1$, $dist_B(v_i,v_{i+1})=2$,
and the \textit{distance of a $2n$-walk}, $P=(v_0,v_1,\ldots,v_n)$, is defined to be $dist(P):=dist_B(v_0,v_n)$.

Fix an $i$-type vertex $v_0$ as the origin. A $2$-walk $(v,u)$ is called a \textit{1-backstep} (resp. \textit{2-backstep}) w.r.t. $v_0$,
if $ dist_B(v_0,u)=dist_B(v_0,v)$ (resp. $ dist_B(v_0,u)=dist_B(v_0,v)-2$).
\end{dfn}

In the above notations, we get.
\begin{lem} \label{B-in-building}
Let $i\in [d]$ be a fixed type, $B=B_i$ the $i$-induced bipartite graph, and fix an $i$-type vertex $v_0$ as the origin. Then the following holds:\\
1) The bipartite graph $B$, is a $(Q,q+1)$-biregular connected graph, where $Q=Q(q,d,i)$ and $q < Q \leq q^{N_d}$.\\
2) Let $(v,u)$ be a $2$-walk which is a $2$-backstep (w.r.t. $v_0$). Then any apartment which contains $v_0$ and $v$, also contains $u$.\\
3) Let $(v,u)$ be a $2$-walk which is a $1$-backstep (w.r.t. $v_0$), and let $\sigma$ be the wall separating $v$ and $u$.
Let $A$ be an apartment which contains $v_0$ and $v$, and let $A'$ be an apartment which contains $v_0$ and $\sigma$.
Then, either $\sigma \in A$, or $u \in A'$.\\
4) For $j=1,2$, and any $i$-type vertex $v$,
\begin{equation}
\frac{\#( 2\mbox{-walks starting at } v \mbox{, which are } j\mbox{-backstep w.r.t. } v_0)}{\#( 2\mbox{-walks starting at } v)} \leq \left(\frac{N_d}{q}\right)^j
\end{equation}
5) For any $k \leq n \in \mathbb{N}$, 
\begin{equation}
\frac{\#( 2n\mbox{-walks starting at } v_0 \mbox{, of distance } 2k)}{\#( 2n\mbox{-walks starting at } v_0)} \leq \frac{(3N_d)^n}{q^{n-k}}
\end{equation}
\end{lem}

\begin{proof}
1) By property \ref{reg} of the Bruhat-Tits building, each $i$-type vertex is adjacent to $Q=Q(q,d,i)$ of $i$-cotype walls in $B$, 
and any $i$-cotype wall is adjacent to $q+1$ of $i$-type vertices in $B$, i.e. $B$ is $(q+1,Q)$-biregular bipartite graph.
Again from property \ref{reg}, since each vertex is contained in some wall, $Q \geq q+1$, and since each wall is contained in some chamber, $Q \leq q^{N_d}$.
Also, the building is a chamber complex, i.e. each two chambers are connected by a gallery (see \cite[\S~4]{AB}), hence $B$ is connected.

2) Since $(v,u)$ is a $2$-backstep (w.r.t. $v_0$), then there is a minimal path (w.r.t. the $B$-graph metric), from $v_0$ to $v$ which passes through $u$.
The $B$-graph metric is coarser then the Euclidean metric of $\B$, 
hence there is a minimal gallery from (a chamber containing) $v_0$ to (a chamber containing) $v$ which passes through (a chamber containing) $u$. 
But $A$ contains $v_0$ and $v$, and since each apartment is convex (property \ref{geo}), $A$ also contains $u$.

3) Since $(v,u)$ is a $1$-backstep (w.r.t. $v_0$), and $\sigma$ is the separating wall, then $dist_B(v_0,\sigma)$ is either $dist_B(v_0,v)+1$ or $dist_B(v_0,v)-1$.
If $dist_B(v_0,\sigma)=dist_B(v_0,v)-1$, then there is a minimal gallery from $v_0$ to $v$ which passes through $\sigma$, 
hence by convexity of the apartment (property \ref{geo}), $\sigma$ is contained in $A$.
If $dist_B(v_0,\sigma)=dist_B(v_0,v)+1=dist_B(v_0,u)+1$, then there is a minimal gallery from $v_0$ to $\sigma$ which passes through $u$, 
hence by convexity of the apartment (property \ref{geo}) $u$ is contained in $A'$.

4) By 1) there are $Q\cdot (q+1)$ number of $2$-walks starting from $v$. 

We start by counting how many of them are $2$-backstep w.r.t. $v_0$.
Fix an apartment, $A$, which contains both $v$ and $v_0$. By claim 2), any $2$-backstep (w.r.t. $v_0$) starting from $v$, must be in $A$.
By property \ref{reg}, inside $A$, each vertex is contained in at most $N_d$ chambers and each wall is contained in exactly $2$ chambers,
hence in an apartment, the number of $2$-walks starting from a given vertex is at most $N_d$.
Thus the portion of $2$-backsteps (w.r.t. $v_0$) among all $2$-walks starting at $v$ is 
$\leq \frac{N_d}{Q\cdot (q+1)} \leq \left(\frac{N_d}{q}\right)^2$.

Now let us count the number of $1$-backsteps w.r.t. $v_0$. 
Let $\sigma_1,\ldots,\sigma_Q$ be the $Q$ neighbors (which are $i$-cotype walls) of $v$ in $B_i$.
Let $A_1,\ldots,A_Q$ be apartments, such that $A_t$ contains $v_0$ and $\sigma_t$, and let $A$ be an apartment containing $v_0$ and $v$.
By claim 3), any $1$-backstep (w.r.t. $v_0$) is either contained in one of the $A_t$, or in $A$.
We have shown already that each apartment contains at most $N_d$ number of $2$-walks starting from $v$, 
hence the total number of $1$-backsteps (w.r.t. $v_0$) is at most $(Q+1)\cdot N_d$.
I.e. the portion of $1$-backsteps (w.r.t. $v_0$) among all $2$-walks starting from $v$ is $\leq \frac{(Q+1)\cdot N_d}{Q\cdot (q+1)} \leq \frac{N_d}{q}$.

5) For a $2n$-walk, $P=(v_0,v_1,\ldots,v_n)$, define $\tau_1=\tau_1(P)$ (resp. $\tau_2=\tau_2(P)$) to be the number of 1-backsteps (resp. 2-backsteps) w.r.t. $v_0$ in $P$.
Note that the distance of a $2n$-walk, $P$, is equal to $2n- 2\tau_1(P) - 4\tau_2(P)$.
Denote by $\alpha_{k,n} $ the probability that a random $2n$-walk $P$ is of distance $2k$, 
and consider a random $2n$-walk $P$, as a sequence of $n$ random $2$-walks. Then by claim 4), we get,
$$\alpha_{k,n} = \sum_{2n-2l-4j=2k} Pr_{P}[\tau_1=l \wedge \tau_2=j]
\leq \sum_{l+2j=n-k} {n \choose l,j} \left( \frac{N_d}{q} \right)^{l+2j} \leq 3^n \left( \frac{N_d}{q} \right)^{n-k}$$
where the last inequality follows from $\sum_{l,j} {n \choose l,j} \leq 3^n$, which finishes the proof.
\end{proof}

Next, we collect some properties on the building coming from the action of the group $G$, on the building. 
Before stating these properties, we shall need the following remark.

\begin{rem} \label{G/K}
Let $G_0$ be the group of type preserving elements in $G$, and recall (property \ref{sym}) that $G_0$ acts transitively on same type simplicies in $\B$.
So, if $K$ is the stabilizer of an $i$-type vertex $v_i$, then we can identify the set of $i$-type vertices in $\B$, with the $K$-cosets $G_0/K$.
Moreover, if $Y\subset G_0/K$ is the set of all $i$-type vertices whose distance (in the $B$-metric) from $v_i=eK$ is exactly $2k$, and $x=gK$ is any $i$-type vertex,
then $gY = gKY \subset G_0/K$ is the set of all $i$-type vertices whose distance (in the $B$-metric) from $x=gK$ is exactly $2k$.
\end{rem}

The following, Lemma describes the neighboring relation of the type-induced graphs of the building, in terms of the action of the group $G$.

\begin{lem} \label{neighbour-building}
Let $A$ be an apartment of the building $\B$, and let $C$ be a chamber in $A$. 
Let $i \in [d]$ be a fixed type, and let $v_{sp}$ and $v_i$ be the unique vertices in $C$ of special-type and $i$-type, respectively (it might be that $v_i=v_{sp}$).
For any $k \in \mathbb{N}$, there exists a finite set $\Omega = \Omega_{i,k} \subset G_0$, such that:\\
1)  $\Omega.v_i$ is the set of all $i$-type vertices in $A$ whose $B_i$-distance from $v_i$ is $2k$.\\
2) Let $K=stab_G(v_i)$ be the stabilizer of $v_i$ in $G$. Identify the $i$-type vertices in $\B$ with the $K$-cosets in $G_0/K$.
Then $K\Omega K \subset G_0/K$ is the is the set of all $i$-type vertices whose $B_i$-distance from $v_i=eK$ is exactly $2k$,
and by the above remark, for each $i$-type vertex $gK$, $gK \Omega K$ is the set of all the $i$-type vertices in the building whose $B_i$-distance from $gK$ is $2k$.\\
3) For any $w \in \Omega$, the skeleton-distance (i.e. the distance in the graph metric of the $1$-skeleton of the building) 
from $v_{sp}$ to $w.v_{sp}$, is at least $\frac{2k}{d}$.\\
4) Let $K_0=stab_G(v_{sp})$ be the stabilizer of $v_{sp}$ in $G$, in particular $K_0$ is a good maximal compact subgroup of $G$.
Let $\Lambda$ be the maximal split torus corresponding to the apartment $A$, $\Phi(G,\Lambda)$ the corresponding roots system, 
$\Lambda^+$ the corresponding lattice of weights, and $G_0 = K_0 \Lambda^+ K_0$ the corresponding Cartan decomposition of $G_0$ w.r.t. $K_0$. 
Let $\alpha_{max} \in \Phi(G,\Lambda)$ be the root of maximal height. Then 
$$ \forall w \in \Omega , w = k_1 \cdot a \cdot k_2 \in K_0 \Lambda^+ K_0 = G_0 \; : \; |\alpha_{max}(a)|_F \leq q^{-\frac{2k}{d}}. $$
\end{lem}

\begin{ex}
In the special case of $PGL_d$, and $k=1$, one can take $\Omega_{i,1}$ to be the singleton $\{diag(\pi,1,\ldots,1,\pi^{-1})\}$ for any $i$.
\end{ex}

\begin{proof}
1) We start by constructing $\Omega$:
Let $\mathcal{V}$ be the set of $i$-type vertices in $A$ whose $B_i$-distance from $v_i$ is exactly $2k$.
For each $y \in \mathcal{V}$, pick a chamber $C_y$ in $A$ which contains $y$, and such that its gallery distance from $C$ is maximal among all chambers in $A$ which contains $y$.
Since $G_0$ acts transitively on chambers, for each $y \in \mathcal{V}$, we can pick $w_y \in G_0$ such that $w_y.C = C_y$, define $\Omega = \{w_y | y \in \mathcal{V} \}$.
Since $\Omega \subset G_0$ and $G_0$ preserves types, for each $w=w_y \in \Omega$, $w.v_i$ is the unique $i$-type vertex in $w.C = C_y$, i.e. $w.v_i = y$,
which proves claim (1).

2) By claim (1), $\Omega.v_i$ is the set of all the $i$-type vertices in the apartment $A$, whose $B_i$-distance from $v_i$ is $2k$.
First, let us show that $K \Omega.v_i$ is equal to the set of all $i$-type vertices in the building whose $B_i$-distance from $v_i$ is $2k$.
It will suffice to show that each $i$-type vertex in the building, $y$, whose $B_i$-distance from $v_i$ is $2k$, is in a $K$-orbit of some vertex from $A$.
For such $y$, pick an apartment $A'$ containing both $v_i$ and $y$, 
and a type-preserving isomorphism $\phi : A' \rightarrow A$ which fixes $v_i$ (such an apartment and an isomorphism, always exists by the axioms of the building).
Then by \cite[Proposition~6.6]{AB}, there is $g \in G_0$ such that for any $F \in A'$, $\phi(F) = g.F$. 
In particular $g.v_i = \phi(v_i)=v_i$ so $g\in K=stab_G(v_i)$, and $g.y = \phi(y) \in A$  as required.
Now, in terms of the identification of $i$-type vertices as $K$-cosets, $K\Omega.v_i = K\Omega K$, and by Remark \ref{G/K}, 
for any $i$-type vertex $x=gK$, $gK \Omega K$ is the set of all the $i$-type vertices in the building whose $B$-distance from $x$ is $2k$, which proves claim (2).

3) Let $w=w_y\in \Omega$, let $y_{sp}=w.v_{sp}$ be the special-type vertex in $C_y=w.C$, 
and let $v_{sp}=y_1,\ldots,y_r=y_{sp}$ be a minimal path in the $1$-skeleton of the building.
Choose a sequence of chambers, $C=C_1,\ldots,C_r=C_y$, such that $C_i$ contains the edge $\{y_i,y_{i+1}\}$ for each $i<r$.
Now, any pair of consecutive chambers, $C_i$, $C_{i+1}$, have a common vertex, $y_{i+1}$, i.e. both chambers induce a pair of chambers in the link of $y_{i+1}$.
By \cite[Proposition~4.9]{AB} the link of any vertex is a $(d-1)$-dimensional finite building, 
and by \cite[Proposition~1.57, Corollary~4.34]{AB}  such a building has diameter $d$, i.e. any two chambers in it are of gallery-distance at most $d$. 
Hence, for any $i<r$, there is a gallery of length $k_i\leq d$ from $C_i$ and $C_{i+1}$ (which is a lift of a gallery between their induced chambers in the link),
 $C_i=C_i^1,\ldots,C_i^{k_i}=C_{i+1}$. Concatenating all these galleries, we get a gallery of length $\sum_{i=1}^{r-1} k_i \leq r\cdot d$, from $C$ to $w.C=C_y$.
However, since any gallery of chambers contains in it a $B_i$-path, and by the definition of $C_y$ the gallery-distance between $C$ and $w.C=C_y$ is at least $2k$,
we get that $r\cdot d \geq 2k$, which proves claim (3).

4) First, note that if $g = k_1 a k_2 \in K_0 \Lambda^+ K_0= G_0$ is the Cartan decomposition of $g$, 
then $a.v_{sp}$ is the unique special-type vertex in the apartment $A= A(\Lambda)$ which is in the $K_0$-orbit of $g.v_{sp}$.
Second, if $z_1,\ldots,z_d \in \Lambda^+$ is a basis of simple weights in $\Lambda^+$,  
then $z_t.v_{sp}$ is of skeleton-distance $2$ from $v_{sp}$, and $\alpha_{max}(z_t) = \pi^2$, for any $t=1,\ldots,d$.
In particular for any $a = \sum_{t=1}^d n_t \cdot z_t \in \Lambda^+$, 
$$|\alpha_{max}(a)|_F \leq q^{-2\cdot \sum_{t=1}^d |n_t|} \leq q^{-dist_{skl}(a.v_{sp},v_{sp})},$$
where $dist_{skl}$ is the skeleton-distance.
So,  for any $w \in \Omega$, if $w = k_1 a k_2 \in K_0 \Lambda^+ K_0= G_0$ is its Cartan decomposition, $a = \sum_{t=1}^d n_t \cdot z_t \in \Lambda^+$, 
then $w.v_{sp} = a.v_{sp}$ and by claim (3) we get
$$ |\alpha_{max}(a)|_F  \leq q^{-dist_{skl}(a.v_{sp},v_{sp})} = q^{-dist_{skl}(w.v_{sp},v_{sp})} \leq q^{-\frac{2k}{d}}.$$
which proves claim (4).
\end{proof}

\section{Second Largest Eigenvalue and Oh's Theorem} \label{spec}
The object of this section is to give a bound on the normalized second largest eigenvalue, $\tilde{\lambda}$, 
of the type-induced bipartite graphs of the finite quotient $\GB$ of the building. 
This will be done by introducing the concept of Hecke operators, which will form a bridge between the representation data ($L^2(\GG)$),
and the combinatorial data $(\tilde{\lambda})$. Applying then Oh's Theorem will give us the needed bound on the second largest eigenvalues of the type induced graphs.

Throughout this section $\GB$ is a finite quotient of the Bruhat-Tits building, which inherit the type-function coming from the building, i.e. $\Gamma \subset G_0$.
Fix once and for all, a type $i \in [d]$, $B=B_i$ the type-induced bipartite graph of $i$-vertices, $V=V_i$, versus $i$-cotype walls, $W=W_i$.

\subsection{distance adjacency operators}
For each $k \in \mathbb{N}$ define the operator $A_k : L^2(V) \rightarrow L^2(V)$ to be the normalize averaging operator on all the vertices of distance exactly $2k$ 
in the bipartite graph $B$, more explicitly 
$$A_kf(x) = \frac{1}{|\{y \in V | dist_{B}(y,x) = 2k\}|} \sum_{dist_{B}(y,x) = 2k} f(y) .$$
Call $A_k$ the normalize $2k$-distance operator of $B$ (starting from $V$).

\begin{pro} \label{B->dist}
If $\Gamma$ is type-preserving of injective radius $r(\Gamma) > 4n$, for some $n \in \mathbb{N}$, then
\begin{equation}
(\tilde{\lambda}(B))^{2n} \leq \sum_{k=0}^n (3N_d)^n \cdot q^{k - n} \cdot \|A_k \|_{L_0^2(V)}
\end{equation}
where $\| \cdot \|$ is the operator norm $L^2_0(V)$ is the subspace perpendicular to the constant function $1_{V}$,
and $N_d$ is defined in the previous section.
\end{pro}

\begin{proof}
Since $B$ is bipartite, its normalized adjacency operator is a matrix of the form 
$\left( \begin{array}{cc} 0 & T \\ T^t & 0 \end{array} \right)$, where $T$ is an operator from $ L^2(V) $ to $ L^2(W)$, and $T^t$ is its transpose.
Therefore $\tilde{\lambda}(B) = \sqrt{\|T^tT\|_{L^2_0(V)}}$. Now, $T^tT$ is the normalized adjacency operator of the 2-walk in $B$, starting from a vertex in $V$.
Therefore, $(T^tT)^n$ is the  normalized adjacency operator of the $2n$-walk in $B$, starting from a vertex in $V$.
So, by the definition of $A_k$, 
\begin{equation}
(T^tT)^n = \sum_{k=0}^n \alpha_{k,n} A_k
\end{equation} 
where $\alpha_{k,n}$ is the portion of $2n$-walks which ends at distance exactly $2k$, among all such $2n$-walks, starting from some $i$-type vertex in the quotient
(a priori, $\alpha_{k,n}$ should depends on the choice of the starting $i$-type vertex, however we shall see shortly that this is not the case).

By Lemma \ref{B-in-building}, in the building, $\alpha_{k,n} \leq (3N_d)^n \cdot q^{k - n}$ (independently of the choice of the starting $i$-type vertex !),
and since we are dealing with walks of length $\leq 2n$ in $\GB$, and since $r(\Gamma) > 4n$, 
i.e. up to distance $2n$ the quotient looks exactly like the covering building, the same claim holds for the quotient.
\end{proof}

\subsection{Hecke operators}
Let $G_0 \leq G$ be the type-preserving subgroup, and let $K \leq G_0$ be the stabilizer of an $i$-type vertex. 
Recall that we can identify the set of $i$-type vertices in $\B$ with $G_0/K$.
For any $w \in G_0$, define the normalize $w$-Hecke operator $H_w : L^2(G_0 /K) \rightarrow L_2(G_0 /K)$, 
$$H_wf(xK) = \frac{1}{|KwK/K|} \sum_{yK \in xKwK} f(yK).$$
The action (from the left) of $\Gamma \leq G_0$ on $L_2(G_0 /K)$ commutes with the action $H_w$.
Hence, $H_w$ can be thought of as a map $H_w : L^2(V) \rightarrow L^2(V)$, where $V=V_i$ is the set of $i$-type vertices in $\GB$, so
$$H_wf(\Gamma xK) = \frac{1}{|KwK/K|} \sum_{\Gamma yK \in \Gamma xKwK} f(\Gamma yK).$$

The next claim gives an interpretation of the normalize $2k$-distance operator in terms of Hecke operators.
\begin{lem} \label{dist->hecke}
Let $k \in \mathbb{N}$ such that $r(\Gamma) > 4k$ and let $\Omega = \Omega_{i,k}$ be as in Lemma \ref{neighbour-building}. 
Then the normalized $2k$-distance operator $A_k$, of the $i$-induced bipartite graph of $\GB$, is a convex sum of $w$-Hecke operators over $w \in \Omega$, in particular
\begin{equation}
\|A_k \|_{L_0^2(V)} \leq \max_{w \in \Omega} \| H_w\|_{L_0^2(V)}
\end{equation}
\end{lem}

\begin{proof}
Under the identification of $i$-type vertices in the building with cosets in $G_0/K$, by Lemma \ref{neighbour-building} (2) we get that:
For any $i$-type vertex $gK$ in the building, the set of $i$-type vertices in the building of distance $2k$ (in $B=B_i$) from $gK$, is equal to $gK\Omega K$. 
Moreover, since $r(\Gamma) > 4k$, up to distance $2k$ the quotient $\GB$ looks exactly like the covering building, so 
For any $i$-type vertex $\Gamma gK \in V$ in quotient, the set of vertices in $V$ of distance $2k$  from $\Gamma gK$, is equal to $\Gamma gK\Omega K$. 
Therefore, if we take $\Omega' \subset \Omega$ a set of representatives of $K \backslash \Omega / K$, then
$$A_k f(xK) = \frac{1}{|K\Omega K/K|} \sum_{\Gamma yK \in xK\Omega K} f(\Gamma yK) = \sum_{w\in \Omega'} \frac{|KwK/K|}{|K\Omega K/K|} \cdot H_w f(\Gamma xK).$$
In particular, 
$$\|A_k \|_{L_0^2(V)} \leq \sum_{w\in \Omega'} \frac{|KwK/K|}{|K\Omega K/K|} \cdot \| H_w\|_{L_0^2(V)} \leq \max_{w \in \Omega} \| H_w\|_{L_0^2(V)}.$$
\end{proof}

Let us now consider the $G$-unitary representation $(\rho,L^2(\GG))$, given by $\rho(g)f(x) = f(xg)$ for $g \in G$ and $f \in L^2(\GG)$.
It turns out that the normalized Hecke operators on $\GB$ are connected to the matrix coefficients of $L^2(\GG)$ in the following way:
let us identify $L^2(V) \cong L^2(\GG_0/K)$ as the subspace of $K$-invariant functions in $L^2(\GG_0)$, and then extend them by zeros to all of $\GG$.
Take on $\GG$ the $G$-invariant measure induced from the Haar measure $\mu$ of $G$ with $\mu(K)=1$, which turn unit vectors in $L^2(V)$ into unit vectors in $L^2(\GG)$.

The following three lemmas appears in \cite[\S~4.2]{EGL}. There it was specialized to the case $G=PGL_d(F)$.
However, the proofs there holds also in our more general setting.

\begin{lem} \cite[Lemma~4.2]{EGL} \label{hecke1}
Let $f_1,f_2 \in L^2(V)$ be two vectors, considered as two $K$-invariant vectors in $L^2(\GG)$ and $w \in G_0$. Then
\begin{equation}
\langle H_w f_1, f_2 \rangle = \langle \rho(w) f_1, f_2 \rangle  
\end{equation}
\end{lem}

\begin{rem} \label{rem-4.3}
Let $G^+ \leq G$ be the subgroup generated by unipotent elements. Let $v$ be any fixed vertex in the building.
Then $G_0 = G^+ \cdot K$, where $K$ is the stabilizer of $v$.
Indeed, by the structure of reductive groups, $G= G^+ \cdot \Lambda$, where $\Lambda$ is any maximal split torus, and $G_0 = G^+ \cdot (\Lambda \cap G_0)$. 
Let $I \subset K$ be an Iwahori subgroup (= point-wise stabilizer of a chamber containing $v$), 
and take $\Lambda \leq G$ to be a maximal torus, such that $I$ has an Iwahori decomposition w.r.t. $\Lambda$, in particular $\Lambda_0:=\Lambda \cap G_0 \subseteq I$. 
Then $G_0 = G^+ \cdot (\Lambda \cap G_0) \subseteq G^+ \cdot I \subseteq G^+ \cdot K \subseteq G_0$.
\end{rem}

Combined with the above remark, the proof of Lemma~4.3 in \cite{EGL}, holds for any reductive group, $G$, and any stabilizer of a vertex in the building, $K \leq G$.
\begin{lem}\cite[Lemma~4.3]{EGL} \label{hecke2}
Let $f \in L^2_0(V)$ and consider it as a $K$-invariant vectors in $L^2(\GG)$. 
Then $f$ is orthogonal to the space of $G^+$-invariant vectors in $L^2(\GG)$.
\end{lem}

From Lemmas \ref{hecke1} and \ref{hecke2} we get,
\begin{cor} \cite[Corollary~4.4]{EGL} \label{hecke3}
Let $U^1$ be the subset of $L^2(\GG)$ of $K$-invariant unit vectors which are orthogonal to any $G^+$-invariant vector. Then,
\begin{equation}
\| H_w \|_{L_0^2(V)} \leq \sup_{f \in U^1} \langle \rho(w) f, f \rangle  
\end{equation} 
\end{cor}

\subsection{Oh's Theorem}
The following result by Oh \cite{Oh}, gives a unified bound on the matrix coefficients of a unitary representation of a reductive group over a local field.

\begin{thm}\cite[Theorem~1.1]{Oh} \label{oh}
Let $F$ be a local non-archimedean field with $char(F) \ne 2$.
Let $G$ be the group of the $F$-rational points of a connected reductive group $\tilde{G}$ over $F$ of rank $\geq 2$ 
and assume $\tilde{G}/Z(\tilde{G})$ is almost $F$-simple. Let $G^+$ be the the subgroup of $G$ generated by the unipotent elements of $G$.

Let $\Phi$ be a root system of $\tilde{G}$ with respect to some maximal torus $\Lambda$, and $\Phi^+ \subset \Phi$ a set of positive roots.
Let $S \subset \Phi^+$ be a strongly orthogonal system of roots, which, by definition, means $\forall \alpha,\beta \in S \, \Rightarrow \alpha \pm \beta \not \in S$.

Let $K_0$ be a good maximal compact subgroup of $G$, i.e. $K_0$ is the stabilizer of a special vertex in the building of $G$.
Any good maximal compact subgroup give rise to a Cartan decomposition $G = K_0 \Lambda^+ K_0$, where $\Lambda^+$ is a positive Weyl chamber 
(= the weight lattice of a maximal split torus $\Lambda$).

Then for any unitary representation $\rho$ of $G$ without a non-zero $G^+$-invariant vectors, 
and for any $K_0$-finite unit vectors $v$ and $u$, 
\begin{equation}
|\langle \rho(g)v,u \rangle| \leq (\dim(K_0v) \dim(K_0u))^{\frac{1}{2}} \xi_S(\lambda)
\end{equation}
where $g=k_1 \lambda k_2\in K_0 \Lambda^+ K_0 = G$, 
$\xi_S(\lambda) = \prod_{\alpha \in S} \Xi_{PGL_2(F)}\left( \begin{array}{cc} \alpha(\lambda) & 0 \\ 0 & 1 \end{array} \right)$ 
and $\Xi_{PGL_2(F)}$ is the Harish-chandra $\Xi$-function of $PGL_2(F)$.
\end{thm}

\begin{rem}
Let us note that for any (unitary) $G$-representation, $\pi$, the $G^+$-invariant part of $\pi$, is a sub-$G$-representation, 
and therefore its orthogonal complement is also a sub-$G$-representation.
Indeed, since $G$ is reductive, then $G=G^+ \cdot Z(G)$ where $Z(G)$ is the center of $G$. 
So, for any $g \in G$, written as $g=g^+\cdot z$ for $g^+\in G^+$ and $z\in Z(G)$, any $G^+$-invariant vector, $v$, and any $h\in G^+$, we get
$h.(g.v)=(h\cdot g^+\cdot z).v=(z\cdot h\cdot g^+).v=z.v=(z\cdot g^+).v=g.v$,
which proves that $g.v$ is also a $G^+$-invariant vector.
\end{rem}

Oh's Theorem gives the following uniform bound on the non-trivial spectrum of the $w$-Hecke operators, 
where $w \in \Omega$ (see Lemma \ref{neighbour-building}).

\begin{cor} \label{hecke->bound}
For any type $i \in [d]$, and any $w \in \Omega_{i,k}$, let $H_w$ be the normalized $w$-Hecke operator acting $L_0^2(V)$. Then,
\begin{equation}
\|H_w \|_{L_0^2(V)} \leq (2k+1) \cdot q^{N_d} \cdot q^{-\frac{k}{d}}
\end{equation}
where $N_d$ is the maximal size of a spherical Coxeter group of rank $d$ (see \S 3).
\end{cor}

\begin{proof}
Let $G$ be as before, $K$ the stabilizer of an $i$-type vertex and $K_0$ the stabilizer of a special vertex.
Consider the subrepresentation of $L^2(\GG)$, which is orthogonal to the $G^+$-invariant part (this is indeed a sub-$G$-representation by the above Remark). 
Applying Theorem \ref{oh} for this representation, together with Corollary \ref{hecke3}, for any strongly orthogonal system $S$, gives us
\begin{equation}
\|H_w \|_{L_0^2(V)} \leq \sup_{f_1,f_2} (\dim(K_0f_1)\dim(K_0f_2))^{1/2} \xi_S(a_w)
\end{equation}
where $f_1,f_2$ are $K$-invariant unit vectors and $w \in K_0 a_w K_0$, $a_w \in \Lambda^+$.

For any $K$-invariant vector $f$, $\dim(K_0 f) \leq [K_0 : K_0 \cap K]$.
Each chamber in $\B$ contains all $[d]$ types, so w.l.o.g. we may choose $K$ and $K_0$ to be the stabilizers of two adjacent vertices $v_i$ and $v_{sp}$, with the possibility that $v_i = v_{sp}$.
Hence, $ [K_0 : K_0 \cap K]$ is bounded by the number of neighbors of $v_i$ of special-type, which is bounded by the number of chambers containing $v_i$, which is $\leq q^{N_d}$ 
(see \S 3, properties of Bruhat-Tits buildings). 
Therefore for any $K$-invariant $f_1$ and $f_2$, 
\begin{equation}
(\dim(K_0 f_1)\dim(K_0 f_2))^{1/2} \leq q^{N_d}
\end{equation} 

Finally, let us take as our strongly orthogonal system the singleton $S=\{\alpha_{max} \}$, where $\alpha_{max}$ is the highest root in $\Phi(G,\Lambda)$. 
Then by Lemma \ref{neighbour-building} (4), the $F$-valuation of $\alpha_{max}(a_w) $ is at least $\frac{2k}{d}$,
and the claim follows from the following $\Xi$-function formula for $PGL_2(F)$ (see \cite[Section~3.8]{Oh}): For any $n \in \mathbb{N}$,
\begin{equation}
\Xi_{PGL_2(F)} \left( \begin{array}{cc} \pi^{\pm n} & 0 \\ 0 & 1 \end{array} \right) = q^{-n/2}\left(\frac{n(q-1)+q+1}{q+1}\right) \leq (n+1)q^{-n/2}
\end{equation}
where $\pi$ is a uniformizer parameter of the local field $F$.
\end{proof}

\begin{rem}
Let us point out that in the special case where $G=PGL_n$, then any vertex is a special vertex, in which case one can erase the term $q^{N_d}$ in the above Corollary.
In particular, if $G=PGL_n$, then for any non-trivial (i.e. $w\not \in K$) Hecke operator, $\|H_w \|_{L_0^2(V)} \leq \frac{3}{q}$. 
However, for general reductive groups, one only gets the following trivial bound $\|H_w \|_{L_0^2(V)} \leq 3\cdot q^{N_d-\frac{1}{d}}$, for any non-trivial $w$.
This is where the original approach of \cite{FGLNP}, fails for general Bruhat-Tits buildings.
\end{rem}

\subsection{Proofs}
Let us now collect all the results from before, to get a bound on the normalized second largest eigenvalue of type-induced bipartite graphs.
Afterward, the proof of the main Theorem will follow.

\begin{thm} \label{B->bound}
For any $2 \leq d \in \mathbb{N}$ and $n \in \mathbb{N}$, there exist $M_{d,n} >0$, such that:

Let $F$ be a local non-archimedean field with $char(F) \ne 2$, of residue degree $q$.
Let $G$ be the group of the $F$-rational points of a connected reductive group $\tilde{G}$ over $F$ of rank $d$, and assume $\tilde{G}/Z(\tilde{G})$ is almost $F$-simple.
Let $\B$ be the Bruhat-Tits building associated to $G$.
Let $\Gamma \leq G$ be a cocompact lattice, which acts type-preserving on the building, and of injectivity radius $r(\Gamma) > 4n$.
For $i=0,\ldots,d$, let $B_i$ be the $i$-induced bipartite graph of the quotient $\GB$. Then
\begin{equation}
\max_{i=0,\ldots,d}(\tilde{\lambda}(B_i)) \leq M_{d,n} \cdot q^{-\frac{1}{2d} + \frac{N_d}{2n}}
\end{equation}
\end{thm}

\begin{proof}
Let $i$ be any type. By Proposition \ref{B->dist}, 
\begin{equation}
(\tilde{\lambda}(B_i))^{2n} \leq \sum_{k=0}^n (3N_d)^n \cdot q^{k-n} \cdot \|A_k \|_{L_0^2(V)} \leq \sum_{k=0}^n (3N_d)^n \cdot q^{\frac{k-n}{d}} \cdot \|A_k \|_{L_0^2(V)}
\end{equation}
where the last inequality follows from the fact that $q^{k-n}\leq 1$, therefore $q^{k-n} \leq \sqrt[d]{q^{k-n}}=q^{\frac{k-n}{d}}$.
From Lemma \ref{dist->hecke} and Corollary \ref{hecke->bound}, we get
\begin{equation}
\|A_k \|_{L_0^2(V)} \leq \max_{w \in \Omega_{i,k}} \|H_w\|_{L_0^2(V)} \leq (2k+1) \cdot q^{N_d - \frac{k}{d}}
\end{equation}
Combining these two equations together we get,
\begin{equation}
\tilde{\lambda}(B_i) \leq \sqrt[2n]{\sum_{k=0}^n (3N_d)^n \cdot (2k+1) \cdot  q^{\frac{k-n}{d} + N_d - \frac{k}{d}}}  =  M_{d,n} \cdot q^{-\frac{1}{2d} + \frac{N_d}{2n}} 
\end{equation}
where $M_{d,n} = \sqrt[2n]{\sum_{k=0}^n (3N_d)^n \cdot(2k+1)}$, and recall that $N_d$ depends only on $d$ (see \S 3), which finishes the proof.
\end{proof}

We are now able to prove our main theorem.
\begin{proof}[Proof of Theorem \ref{main}]
Let $H$ be the corresponding partite hypergraph, of the finite quotient $X = \GB$. 
We will choose $n=2d\cdot N_d$, and then by Proposition \ref{disc->B} and Theorem \ref{B->bound}, we get
\begin{equation}
Disc(H) \leq d \cdot max_{i=0,\ldots,d}(\tilde{\lambda}(B_i)) 
\leq d \cdot M_{d,2d\cdot N_d} \cdot q^{-\frac{1}{2d} + \frac{N_d}{4d\cdot N_d}} = d \cdot M_{d,2d\cdot N_d} \cdot q^{-\frac{1}{4d}}
\end{equation}
Now, we take $\epsilon = \frac{1}{2}\cdot c_d^{d+1}$ and $q_0(d) := (2d \cdot M_{d,2dN_d} \cdot (c_d)^{-(d+1)})^{4d}$ (where $c_d$ is Pach's constant),
and note that $\epsilon$ and $q_0$ depends only on $d$. Then for any $q> q_0$ ,$(c_d)^{d+1} - Disc(H) \geq \epsilon > 0$.
Therefore, by Corollary \ref{over->disc} the quotient $X=\GB$ has the $\epsilon$-geometric overlap property.
\end{proof}

\begin{rem}
By \cite{BH} any reductive group over zero characteristic local field has cocompact arithmetic lattices,
and by a standard argument they have finite index subgroups satisfying the assumption of Theorem \ref{main}.
\end{rem}


\end{document}